\documentclass[a4paper,12pt,reqno]{amsart}

\makeindex

\usepackage{amsfonts,graphics,amsmath,amsthm,amsfonts,amscd,amssymb,amsmath,latexsym,euscript, enumerate}
\usepackage{epsfig,url}
\usepackage{flafter}
\usepackage[all,cmtip,line]{xy}
\usepackage{array}
\usepackage[english]{babel}
\usepackage{overpic}
\usepackage{subfig}
\usepackage{multirow}

\usepackage{wrapfig}
\usepackage[shortlabels]{enumitem}
\setlist[enumerate, 1]{$(1)$}

\usepackage[usenames,dvipsnames,svgnames,table]{xcolor}
\usepackage{graphicx}
\usepackage[bookmarks=true,bookmarksnumbered=true,
 pdftitle={},
 pdfsubject={},
 pdfauthor={},
 pdfkeywords={TeX, dvipdfmx, hyperref, color,},
 colorlinks=true, linkcolor=RoyalBlue, citecolor=Emerald]
 {hyperref}

\usepackage{cancel}

\newtheorem{theorem}{Theorem}[section]
\newtheorem{lemma}[theorem]{Lemma}
\newtheorem{proposition}[theorem]{Proposition}

\newtheorem{corollary}[theorem]{Corollary}

\theoremstyle{definition} 
\newtheorem{definition}[theorem]{Definition}
\newtheorem{definition-lemma}[theorem]{Definition-Lemma}

\newtheorem{example}[theorem]{Example}

\newtheorem{remark}[theorem]{Remark}

\numberwithin{equation}{section}

\newcommand{\C}{\mathbb{C}}
\newcommand{\R}{\mathbb{R}}
\newcommand{\Z}{\mathbb{Z}}

\newcommand{\Q}{\mathbb{Q}}
\newcommand{\mc}{\mathcal}
\def\P{\mathbb{P}}

\newcommand{\mf}{\mathfrak}

\DeclareMathOperator{\ord}{ord}

\DeclareMathOperator{\Mob}{Mob}
\DeclareMathOperator{\Fix}{Fix}

\DeclareMathOperator{\id}{id}
\DeclareMathOperator{\QM}{QM}
\DeclareMathOperator{\Val}{Val}

\newcommand{\pp}{P_{\sigma}}
\newcommand{\np}{N_{\sigma}}

\def\mult{\operatorname{mult}}

\def\Supp{\operatorname{Supp}}
\def\Exc{\operatorname{Exc}}
\def\NE{\operatorname{NE}}

\def\lct{\operatorname{lct}}

\def\pnklt{\operatorname{pNklt}}
\def\nklt{\operatorname{Nklt}}

\usepackage{mathtools}

\DeclarePairedDelimiterX{\inp}[2]{\langle}{\rangle}{#1, #2}
\DeclarePairedDelimiterX{\norm}[1]{\lVert}{\rVert}{#1}

\DeclarePairedDelimiter\floor{\lfloor}{\rfloor}

\title
{Asymptotic multiplier ideal sheaves associated to potential triples}

\begin{document}

\author{Sung Rak Choi}
\author{Sungwook Jang}
\author{Donghyeon Kim}
\address{Department of Mathematics, Yonsei University, 50 Yonsei-ro, Seodaemun-gu, Seoul 03722, Republic of Korea}
\email{sungrakc@yonsei.ac.kr}
\address{Center for Complex Geometry, Institute for Basic Science, 34126 Daejeon, Republic of Korea}
\email{swjang@ibs.re.kr}
\address{Department of Mathematics, Yonsei University, 50 Yonsei-ro, Seodaemun-gu, Seoul 03722, Republic of Korea}
\email{whatisthat@yonsei.ac.kr, narimial0@gmail.com}

\thanks{The authors are partially funded by Samsung Science and Technology Foundation under Project Number SSTF-BA2302-03.}

\date{\today}

\begin{abstract}
In this paper, we explore the geometry of potential triples $(X,\Delta,D)$, which by definition consists of a pair $(X,\Delta)$ and an $\R$-Cartier pseudoeffective divisor $D$ on $X$. We define and study the asymptotic multiplier ideal sheaf $\mc J(X,\Delta,\norm{D})$ associated to a potential triple $(X,\Delta,D)$. As a first main result, when $D$ is big, we prove that the condition $\mc{J}(X,\Delta,\norm{D})=\mc{O}_X$ is equivalent to the triple $(X,\Delta,D)$ being potentially klt, which is a klt analog of the pair $(X,\Delta)$. We also study the closed set defined by the ideal sheaf $\mc{J}(X,\Delta,\norm{D})$ and prove a Nadel type cohomology vanishing theorem for $\mc J(X,\Delta,\norm{D})$. As an application of the main result, we  prove that we can run the $(K_X+\Delta+D)$-MMP with scaling of an ample divisor for a pklt triple $(X,\Delta,D)$.
\end{abstract}

\maketitle

\section{Introduction}\label{sect:1}

In birational geometry, understanding the intertwined relation between the singularities and positivity given by the divisors is inevitable. For instance, the recent advancement of the birational geometry including the minimal model program and the settlement of the BAB conjecture are undoubtedly the products of the deep results that unveil how singularities and positivity of divisors affect the geometry of the pairs. In this paper, we elaborate on how the local and global geometric properties of divisors can be studied in a common framework that we describe below.

For a pair $(X,\Delta)$ (i.e., $X$ is a normal projective variety and $\Delta$ is an effective $\R$-divisor $\Delta$ such that $K_X+\Delta$ is $\R$-Cartier), the multiplier ideal sheaf $\mc{J}(X,\Delta)\subseteq\mc O_X$ detects the singularities of the pair $(X,\Delta)$. More precisely, a pair $(X,\Delta)$ is klt if and only if $\mc{J}(X,\Delta)=\mc O_X$ and if the strict inclusion $\mc{J}(X,\Delta)\subsetneq\mc O_X$ holds, then $\mc{J}(X,\Delta)$ precisely defines the non-klt locus $\nklt(X,\Delta)$ of $(X,\Delta)$. On the other hand, for a smooth projective variety $X$ and a big Cartier divisor $D$ on $X$, the asymptotic multiplier ideal sheaves $\mc{J}(X,\norm{m D})$ $(m\in\mathbb{Z}_{>0})$ detect the nefness of $D$. More precisely, $D$ is nef if and only if $\mc{J}(X,\norm{mD})=\mc{O}_{X}$ for all positive integers $m$ (\cite[11.2.18]{laza2}) and the union of the closed sets defined by $\mc{J}(X,\norm{mD})$ for all $m>0$ coincides with the non-nef locus $\textrm{NNef}(D)$ of $D$.  However, it is not known clearly what information a single ideal sheaf $\mc{J}(X,\norm{D})$ carries.
We first construct and study the ideal sheaf which reflects both the aforementioned local and global properties simultaneously.

A \emph{potential triple} $(X,\Delta,D)$ consists of a pair $(X,\Delta)$ and a pseudoeffective $\R$-Cartier divisor $D$ on $X$ (Definition \ref{def:pklt}).  We will often simply call it a \emph{triple}. We note that the notion of a triple is more general than that of a generalized pair since $D$ in a triple $(X,\Delta,D)$ is only a pseudoeffective $\R$-divisor, not necessarily a b-nef divisor. As we will see, many of the notions defined for pairs and generalized pairs also extend to triples (Section \ref{sect:2}).

Let $(X,\Delta,D)$ be a triple and $E$ be a prime divisor over $X$. We define the \emph{log discrepancy} $a(E;X,\Delta,D)$ as $a(E;X,\Delta,D):=A_{X,\Delta}(E)-\sigma_{E}(D)$, where $A_{X,\Delta}(E)$ is the log discrepancy of the pair $(X,\Delta)$ along $E$ and $\sigma_{E}(D)$ is the asymptotic divisorial valuation of $D$ along $E$. If $\inf_{E}a(E;X,\Delta,D)>0$ where the infimum $\inf$ is taken over all prime divisors $E$ over $X$, then we say that $(X,\Delta,D)$ is \textit{potentially klt}. We use the abbreviation \emph{pklt} for potentially klt. See Section \ref{sect:2} for the precise definitions.

Let $(X,\Delta,D)$ be a triple where $D$ is an effective $\Q$-Cartier $\Q$-divisor on $X$. Let $p$ be a positive integer and $f:Y\to X$ be a resolution of the linear system $|pD|$ for which there exists a decomposition $f^*|pD|=|M_p|+F_p$ into the free part $|M_{p}|$ and the fixed part $F_p$ with the simple normal crossing support. The family $\{f_*\mc O_Y(K_{Y}-\floor{f^*(K_{X}+\Delta)+\frac{1}{p}F_p})|\;p\in\mathbb Z_{>0}\}$ of the ideal sheaves contains the unique maximal element by Lemma \ref{lem:maximal ideal}. We denote it by $\mc{J}(X,\Delta,\norm{D})$ and call it the asymptotic multiplier ideal sheaf associated to the triple $(X,\Delta,D)$. Note that the ideal sheaves $\mc{J}(X,\Delta)$ and $\mc{J}(X,\norm{D})$ appear as special cases of $\mc{J}(X,\Delta,\norm{D})$ for any semiample divisor $D$ and $\Delta=0$,  respectively. The following is the first main result of this paper.

\begin{theorem}\label{maintheorem1}
Let $(X,\Delta,D)$ be a triple with a big $\Q$-divisor $D$ on $X$. Then the following are equivalent:
\begin{enumerate}[wide = 0pt, leftmargin = 1.4em]
\item the triple $(X,\Delta,D)$ is pklt.
\item $\mc{J}(X,\Delta,\norm{D})=\mc O_X$.
\item there exists an effective divisor $D'\sim_{\R}D$ such that $(X,\Delta+D')$ is klt.
\end{enumerate}
\end{theorem}

The effective divisor $D'$ in $(3)$ is called a \textit{klt complement} of the triple $(X,\Delta,D)$. See Definition \ref{def:p-complement}.
The proof of Theorem \ref{maintheorem1} given in Section \ref{sect:3} depends on the results from the valuation theory which is developed in \cite{JM}, \cite{Xu}. This replaces the proof of \cite[Proposition 4.4]{CP} (Remark \ref{remk:CP 4.4}).

\medskip

Next we consider the case $\mc{J}(X,\Delta,\norm{D})\subsetneq \mc O_X$.
The \emph{potentially non-klt locus} $\pnklt(X,\Delta,D)$ of a triple $(X,\Delta,D)$ is defined as the union of the centers $c_X(E)$ of the prime divisors $E$ over $X$ such that $a(E;X,\Delta,D)\leq 0$ (Definition \ref{def:pnklt locus}). By \cite{Nak}, if $D$ is nef, then $\sigma_E(D)=0$ for all $E$. Therefore, if $D$ is nef, then $\pnklt(X,\Delta,D)$ coincides with the non-klt locus $\nklt(X,\Delta)$ of the pair $(X,\Delta)$.

By construction, it is easy to see that the center $c_X(E)$ of a prime divisor $E$ over $X$ with $a(E;X,\Delta,D)\leq 0$ is contained in $\mc{Z}(\mc{J}(X,\Delta,\norm{D}))$. Thus we have the following inclusion in general (Proposition \ref{prop:center in pnklt})
$$\pnklt(X,\Delta,D)\subseteq\mc{Z}(\mc{J}(X,\Delta,\norm{D})).$$
We prove in Theorem \ref{thrm:J_pnklt eff -K-Delta} that if $D$ is a big $\Q$-divisor which admits a birational Zariski decomposition, then we have the equality $\pnklt(X,\Delta,D)=\mc{Z}(\mc{J}(X,\Delta,\norm{D}))$.

\medskip

As an application of Theorem \ref{maintheorem1}, we prove that for any pklt triple $(X,\Delta,D)$, we can run the $(K_X+\Delta+D)$-MMP with scaling. This MMP can be viewed as a triple analog of the existence of MMP for a projective klt pair $(X,\Delta)$ (cf. \cite[Theorem 3.7]{km}, \cite[Corollary 1.4.1]{bchm}).

\begin{theorem}\label{maintheorem4}
Let $(X,\Delta,D)$ be a pklt triple with a pseudoeffective $\Q$-divisor $D$. Then we can run the $(K_X+\Delta+D)$-MMP with scaling of an ample divisor.
\end{theorem}

\medskip

We also prove the Nadel type vanishing theorem for $\mc{J}(X,\Delta,\norm{D})$.

\begin{theorem}\label{maintheorem3}
Let $(X,\Delta,D)$ be a triple with effective $\Q$-divisors $\Delta$ and $D$ on $X$.
If $L$ is a Cartier divisor such that $L\sim_\Q K_X+\Delta+D$, then
      $$
      H^q(X,\mc{O}_X(L)\otimes\mc{J}(X,\Delta,\norm{D}))=0\;\;\;\text{for all $q>\dim X-\kappa(D)$.}
      $$
\end{theorem}

In this theorem, if $(X,\Delta,D)$ is a pklt triple and $D$ is big, then we obtain the vanishing $H^q(X,\mc{O}_X(L))=0$ for all $q>0$ (Corollary \ref{cor:main cor on vanishing}).

The multiplier ideal sheaves of metrics with minimal singularities are the analytic counterpart of the asymptotic multiplier ideal sheaves. Let $X$ be a smooth projective variety and $L$ a big line bundle on $X$. For a singular metric $h_{\min}$ of $L$ with minimal singularities and nonnegative curvature current, let $\mc{J}(h_{\min})$ be the associated multiplier ideal sheaf. Then by \cite{GZ}, we have $\mc{J}(h_{\min})=\mc{J}(X,\norm{L})$.

In \cite{Cao}, Cao proved the vanishing theorem for a variant of analytic multiplier ideal sheaf with the restriction on the degrees of cohomology groups. Using $L^{2}$-estimate, Wu in \cite{Wu} also proved the vanishing theorem for asymptotic multiplier ideal sheaf associated to a pseudoeffective line bundle on a compact K\"ahler manifold. Our vanishing theorems for asymptotic multiplier ideal sheaves of triples can be considered as their extensions to the pair cases in the algebraic settings.

We also note in Remark \ref{remark: triple more general than gpair} that the generalized pairs in \cite{BZ} can be naturally treated as potential triples  and our results recover some of the results on vanishing of cohomologies for the generalized pairs at least in the generalized klt case or the usual klt case (Remark \ref{remark: potential is general}).

\medskip

This paper is organized as follows: In Section \ref{sect:2}, we introduce the notion of potential triples $(X,\Delta,D)$ which is a key object in this paper. In Section \ref{sect:3}, we define the asymptotic multiplier ideal sheaf $\mc J(X,\Delta,\norm{D})$ associated to a triple $(X,\Delta,D)$. Theorems \ref{maintheorem1} and \ref{maintheorem4} are proved in this section. We also discuss an alternative construction of $\mc J(X,\Delta,\norm{D})$ when $D=-(K_X+\Delta)$. In Section \ref{sect:4}, the vanishing theorems including Theorem \ref{maintheorem3} are proved.

\section*{Acknowledgement}

We thank Dae-Won Lee for various suggestions and comments on earlier versions of this paper.

\bigskip

\section{Preliminaries}\label{sect:2}

\subsection{Pairs and Triples}\label{subsec:pairt_riple}

We work over an algebraically closed field $k$ of characteristic $0$, e.g., the complex number field $\mathbb C$. By a variety $X$, we mean a normal projective variety defined over $k$. Divisors are $\R$-Weil divisors unless otherwise stated. A \emph{pair} $(X,\Delta)$ consists of a variety $X$ and an effective divisor $\Delta$ on $X$ such that $K_X+\Delta$ is $\R$-Cartier.
We first recall two invariants that measure \emph{local} singularities of pairs and  \emph{global} positivity of divisors.

Let $(X,\Delta)$ be a pair and $E$ a prime divisor over $X$. Then there exists a birational morphism $f:Y\to X$ from a normal projective variety $Y$ where $E$ is a prime divisor on $Y$. We can write $K_Y=f^*(K_X+\Delta)+\sum e(E_i;X,\Delta)E_i$ with distinct prime divisors $E_i$ on $Y$. Then $A_{X,\Delta}(E):=e(E;X,\Delta)+1$ is called  the \emph{log discrepancy} of $(X,\Delta)$ along $E$ and it can be computed for any prime divisor $E$ over $X$ by considering a birational morphism $f:Y\to X$ where $E$ is a prime divisor on $Y$. The pair $(X,\Delta)$ is said to have \textit{klt singularities} if $\inf_{E}A_{X,\Delta}(E)>0$, and \textit{lc singularities} if $\inf_{E}A_{X,\Delta}(E)\ge0$ for all prime divisors $E$ over $X$, respectively.

Let $X$ be a smooth projective variety and $D$ a pseudoeffective divisor on $X$. For a prime divisor $E$ on $X$, the \emph{asymptotic valuation} $\sigma_E(D)$ is defined as follows: If $D$ is big, then we define $\sigma_E(D):=\inf\{\mult_ED'| D\sim_{\R}D'\geq 0\}$. If $D$ is only pseudoeffective, then we define $\sigma_E(D):=\lim\limits_{\epsilon\to 0}\sigma_E(D+\epsilon A)$ for an ample divisor $A$ on $X$. This definition is independent of the choice of $A$. It is well-known that there are only finitely many prime divisors $E$ on $X$ such that $\sigma_E(D)>0$ (\cite{Nak}). The \textit{negative part} $\np(D)$ of $D$ is defined as
$$ \np(D):=\sum_{E}\sigma_{E}(D)E $$
and the \textit{positive part} $\pp(D)$ of $D$ is defined as $\pp(D):=D-\np(D)$. We call $D=\pp(D)+\np(D)$ the \textit{divisorial Zariski decomposition} of $D$. If $\pp(D)$ is nef, then we call $D=\pp(D)+\np(D)$ the \textit{Zariski decomposition} of $D$. The following lemma shows the behavior of the negative part under pull-backs.

\begin{lemma}[\protect{\cite[Theorem III.5.16]{Nak}}] \label{lemma-np-behavior}
Let $f:Y\to X$ be a birational morphism of smooth projective varieties and $D$ a pseudoeffective divisor on $X$. Then there exists an effective $f$-exceptional divisor $\Gamma$ on $Y$ such that
$$ \np(f^{\ast}D)=f^{\ast}\np(D)+\Gamma. $$
Furthermore, if $\pp(D)$ is nef, then $\np(f^{\ast}D)=f^{\ast}\np(D)$.
\end{lemma}

The asymptotic valuations naturally extend to the case of normal projective varieties and such extension also allows us to define the (divisorial) Zariski decomposition of $\R$-Cartier divisors on the normal projective varieties as well. Let $X$ be a normal projective variety and $D$ an $\R$-Cartier pseudoeffective divisor on $X$. Suppose that $E$ is a prime divisor over $X$. Then there exists a birational morphism $f:Y\to X$ such that $Y$ is a smooth projective variety and $E$ is a prime divisor on $Y$. We define $\sigma_{E}(D):=\sigma_{E}(f^{\ast}D)=\mult_{E}\np(f^{\ast}D)$. By Lemma \ref{lemma-np-behavior}, $\sigma_{E}(D)$ does not depend on the choice of $f$. We can also define the \textit{positive part} $\pp(D)$ of $D$ as $\pp(D):=f_{\ast}\pp(f^{\ast}D)$. Similarly, the \textit{negative part} $\np(D)$ of $D$ is defined as $\np(D):=f_{\ast}\np(f^{\ast}D)$, and the decomposition $D=\pp(D)+\np(D)$ is called the \textit{divisorial Zariski decomposition} of $D$. By Lemma \ref{lemma-np-behavior}, these definitions do not depend on the choice of $f$.
We say that a pseudoeffective divisor $D$ \textit{admits a birational Zariski decomposition} if there exists a birational morphism $f:Y\to X$ such that $\pp(f^{\ast}D)$ is nef.

\medskip

Next we define a notion which is more general than pairs.

\begin{definition} \label{def:pklt}
 \hfill
\begin{enumerate}[wide = 0pt, leftmargin = 1.4em]
\item A \textit{potential triple} $(X,\Delta,D)$ consists of a pair $(X,\Delta)$ and a pseudoeffective $\R$-Cartier divisor $D$ on $X$.

\item The \textit{log discrepancy} $a(E;X,\Delta,D)$ of a triple $(X,\Delta,D)$ along a prime divisor $E$ over $X$ is defined as
$$ a(E;X,\Delta,D):=A_{X,\Delta}(E)-\sigma_E(D).$$

\item The triple $(X,\Delta,D)$ is said to be \textit{potentially klt} if $\inf_E a(E;X,\Delta,D)>0$ where $\inf$ is taken over all prime divisors $E$ over $X$. The triple $(X,\Delta,D)$ is said to be \textit{weakly potentially klt} if $a(E;X,\Delta,D)>0$ for any prime divisor $E$ over $X$. A triple $(X,\Delta,D)$ is said to be \textit{potentially lc} if $\inf_E a(E;X,\Delta,D)\ge 0$. As usual, we write pklt, plc for potentially klt, potentially lc, respectively.
\end{enumerate}
\end{definition}

Since $A_{X,\Delta}(E)$ and $\sigma_{E}(D)$ depend only on the divisorial valuation given by $E$ for a triple $(X,\Delta,D)$, so does the log discrepancy $a(E;X,\Delta,D)$. Note that in Definition \ref{def:pklt}, by fixing a birational morphism $f:Y\to X$ such that $E$ is a prime divisor on $Y$, $a(E;X,\Delta,D)$ can be also computed as follows: we can write $K_Y+\Delta_Y=f^*(K_X+\Delta)$ for some divisor $\Delta_Y$ on $Y$ and let $f^*D=P+N$ be the divisorial Zariski decomposition. Then we have $a(E;X,\Delta,D)=\mult_E(-\Delta_Y-N)+1$.

\begin{definition}\label{def:pnklt locus}
The \emph{potentially non-klt locus} $\pnklt(X,\Delta,D)$ of a triple $(X,\Delta,D)$ is defined as
$$
\pnklt(X,\Delta,D):=\bigcup_E c_X(E)
$$
where the union is taken over all prime divisors $E$ over $X$ such that $a(E;X,\Delta,D)\leq 0$.
\end{definition}

By definition, $\pnklt(X,\Delta,D)=\emptyset$ if and only if$(X,\Delta,D)$ is weakly pklt.
If $D$ is nef, then $\pnklt(X,\Delta,D)$ coincides with the non-klt locus $\nklt(X,\Delta)$ of the pair $(X,\Delta)$.

\begin{definition}\label{def:p-complement}
For a triple $(X,\Delta,D)$, an effective $\R$-Cartier $\R$-divisor $D'\sim_\R D$ such that $(X,\Delta+D')$ is called a \textit{klt complement} of $(X,\Delta,D)$.
\end{definition}

By definition, pklt triples are weakly pklt. The converse holds in the following cases.

\begin{lemma} \label{lem:weakly pklt}
Let $(X,\Delta,D)$ be a weakly pklt triple. If $(X,\Delta,(1+\varepsilon)D)$ is weakly pklt for some $\varepsilon>0$, then $(X,\Delta,(1+\varepsilon')D)$ is pklt for any $\varepsilon'$ such that $0\le \varepsilon'<\varepsilon$. In particular, $(X,\Delta,D)$ is pklt.
\end{lemma}

\begin{proof}
Let $\alpha=\inf_E\{A_{X,\Delta}(E)\}$ where $\inf$ runs through all prime divisors $E$ over $X$. Since the given condition implies that $(X,\Delta)$ is klt, we have $\alpha>0$. By assumption, we have
\begin{equation*}\tag{$*$}\label{*} 
A_{X,\Delta}(E)-(1+\varepsilon)\sigma_{E}(D)>0
\end{equation*}
for any prime divisor $E$ over $X$. Fix $0\le \varepsilon'<\varepsilon$. If $\sigma_{E}(D)\le \frac{\alpha}{2(1+\varepsilon')}$, then
$$ A_{X,\Delta}(E)-(1+\varepsilon')\sigma_{E}(D)\ge \alpha-\frac{\alpha}{2}=\frac{\alpha}{2}>0. $$
If $\sigma_{E}(D)\ge \frac{\alpha}{2(1+\varepsilon')}$, then by (\ref{*})
\begin{equation*}
A_{X,\Delta}(E)-(1+\varepsilon')\sigma_{E}(D)>(\varepsilon-\varepsilon')\sigma_{E}(D)\ge \frac{\alpha(\varepsilon-\varepsilon')}{2(1+\varepsilon')}>0.
\end{equation*}
Thus for $\beta:=\min\{\frac{\alpha}{2},\frac{\alpha(\varepsilon-\varepsilon')}{2(1+\varepsilon')}\}$, we have 
$$
\inf_E a(E,X,\Delta,(1+\varepsilon')D)\geq \beta>0
$$
which implies that $(X,\Delta,(1+\varepsilon')D)$ is pklt.
\end{proof}

\medskip

\begin{proposition}\label{prop:Zariski resolution}
Suppose that the divisor $D$ in a triple $(X,\Delta,D)$ admits a birational Zariski decomposition. Then the triple $(X,\Delta,D)$ is weakly pklt if and only if it is pklt.
\end{proposition}
\begin{proof}
The proof for the case $D=-(K_X+\Delta)$ is given in \cite[Theorem 2.6]{CJ} and the proof for general $D$ is almost identical with the obvious adjustments. Thus, we skip the proof.
\end{proof}

We also prove a birational property for triples.

\begin{definition}\label{def:D-positive/negative}
Let $f:X\dashrightarrow X'$ be a birational map of varieties, $D$ an $\R$-Cartier divisor on $X$, and let $D'=f_{\ast}D$ be the birational transform of $D$. We say that $f$ is \textit{$D$-nonpositive} (resp. \textit{$D$-negative}) if there exists a common log resolution $g:W\to X$ of $(X,D)$ and $h:W\to X'$ of $(X',D')$ and we can write
$$ g^{\ast}D=h^{\ast}D'+E, $$
where $E$ is an effective $h$-exceptional divisor (resp. an effective $h$-exceptional divisor whose support contains all $f$-exceptional divisors).
\end{definition}

\begin{theorem}[{\cite[Lemma 2.5]{CJL}}]\label{lem:D-MMP invariant}
Let $(X,\Delta,D)$ be a potential triple. Let $\varphi\colon X\dashrightarrow X'$ be either $(K_X+\Delta+D)$-negative divisorial contraction or a flip. Then for $\Delta':=\varphi_*\Delta$ and $D':=\varphi_*D$, we have
$$a(E;X,\Delta,D)\leq a(E;X',\Delta',D') $$
for any prime divisor $E$ over $X$.
\end{theorem}

\medskip

Lastly, we compare the potential triples with the generalized pairs (\cite{BZ}).
Recall that a \textit{generalized pair} $(X,B+M)$ consists of a normal variety $X$, an effective $\R$-divisor $B$ on $X$, a projective birational morphism $f:X'\to X$, and a nef divisor $M'$ on $X'$ such that $K_{X}+B+M$ is $\R$-Cartier, where $M=f_{\ast}M'$.
In the definition, we may assume that $f$ is a log resolution of $(X,B)$. Then for a given generalized pair $(X,B+M)$, there exists a divisor $B'$ on $X'$ such that
$$ K_{X'}+B'+M'=f^{\ast}(K_{X}+B+M). $$
If $\mult_EB'<1$ (resp. $\le 1-\varepsilon$ for some $\varepsilon\geq 0$) for any prime divisor $E$ on $X'$, then we say that a generalized pair $(X,B+M)$ is klt (resp. $\varepsilon$-lc).

\bigskip

\begin{remark}\label{remark: triple more general than gpair}
We note that the notion of potential triples $(X,\Delta,D)$ is more general than that of generalized pairs $(X,B+M)$.

Note first that a generalized pair $(X,B+M)$ with an $\R$-Cartier divisor $M$ can be naturally considered as a triple $(X,B,M)$. Furthermore, we claim that if a generalized pair $(X,B+M)$ has klt singularities ($\epsilon$-lc singularities), then the triple $(X,B,M)$ is pklt ($\epsilon$-plc, respectively). Let $f:X'\to X$ be a log resolution of $(X,B)$ and $M'$ the nef divisor on $X'$ such that $f_*M'=M$. There exists a divisor $B'$ on $X'$ such that $K_{X'}+B'+M'=f^*(K_X+B+M)$. We may assume that a prime divisor $E$ over $X$ is a prime divisor on $X'$. Note then that by definition, we have  $a(E;X,B,M)\geq 1-\mult_E(B')$ and this proves the claim. However, the converse does not hold in general. First of all, not all potential triples can be considered as generalized pairs. Such examples can be easily found as follows:
Let $D$ be a pseudoeffective $\R$-Cartier divisor having the Zariski decomposition $D=P+N$ with $N\neq 0$. Then $(X,\Delta,N)$ is a potential triple by definition. However, $(X,\Delta,N)$ is not a generalized pair. Suppose that $(X,\Delta,N)$ is a generalized pair. Then $N=f_*(P')$ for some birational morphism $f:Y\to X$ and a nef divisor $P'$ which is non-zero and effective on the support $\Supp f^{-1}_*N$. Since $f^*D=f^*P+f^*N$ is also the Zariski decomposition and $f^*D=f^*P+P'+E$, we have $f^*P\geq f^*P+P'$. Thus we have $0\geq P'$ and this is a contradiction.
\end{remark}

\begin{example}
The following example also shows that for a potential triple $(X,\Delta,D)$ which can be also viewed as a generalized pair $(X,\Delta+D)$, the notion of pklt for the triple $(X,\Delta,D)$ is also  strictly more general than the notion of klt for the generalized pair $(X,\Delta+D)$.

Suppose that a generalized pair $(X,\Delta+D)$ defines a potential triple $(X,\Delta,D)$ which is pklt. The following example shows that the generalized pair $(X,\Delta+D)$ can have non-klt singularities.
Let $\pi:X\to \P^{2}$ be the blow-up of $\P^{2}$ at a point $x\in \P^{2}$, and $H$ a line passing through the point $x$. Denote by $H'$ the strict transform of $H$. Then $rH'$ is nef for all positive real numbers $r$, and it defines a generalized pair $(\P^{2},0+rH)$. Let $E$ be the exceptional divisor on $X$. Then we have
$$ K_{X}+(r-1)E+rH'=\pi^{\ast}(K_{\P^{2}}+rH). $$
Hence, for $r>2$, the generalized pair $(\P^{2},0+rH)$ is not lc. However, the potential triple $(\P^{2},0,rH)$ is obviously pklt.
\end{example}

\smallskip
\subsection{Valuation}\label{subsec:valuation}\hfill

We recall the notions and the results that will be used in the proof of Theorem \ref{maintheorem1}.
See \cite{JM}, \cite{Xu} for more details.

For a normal variety $X$, a \textit{valuation} on the function field $K(X)$ of $X$ is an $\R$-valued function $$\nu:K(X)^{\times}\to\R$$
such that
\begin{enumerate}
\item $\nu(a)=0$ for any $a\in \C$,
\item $\nu(fg)=\nu(f)+\nu(g)$ for any $f,g\in K(X)$, and
\item $\nu(f+g)\geq \min\{\nu(f),\nu(g)\}.$
\end{enumerate}

For a valuation $\nu$ on $K(X)$, if there is a point $x$ of $X$ which induces a local inclusion $\mathcal{O}_{X,x}\hookrightarrow \mathcal{O}_{\nu}$, then such point $x$ is called the \textit{center} of $\nu$ and we denote it by $c_X(\nu)\coloneqq x$. If the center $c_X(\nu)$ is the generic point of $X$, then $\nu$ is called a \textit{trivial valuation} of $X$.  We denote by $\Val_{X}$ (resp. $\Val^{\ast}_X$) the set of all (resp. nontrivial) valuations on $K(X)$ having centers in $X$. We denote by $\mathrm{Val}_{X,x}\subseteq\mathrm{Val}_{X}$ the subset of all valuations of $X$ whose center is the point $x$ of $X$.

The set $\Val_X$ can be considered as a topological space equipped with the following topology $\tau$. Let $\mathcal{I}$ be a set of nonzero ideals on $X$. To a valuation $\nu\in \Val_X$ and a coherent ideal sheaf $\mf{a}$ on $X$, one can associate a function $\nu\colon \mathcal{I}\rightarrow \R_{\geq 0}$ by defining $\nu(\mf{a})\coloneqq \min\{\nu(f) \mid f\in \mf{a}\cdot \mathcal{O}_{X,x}\}$, where $x=c_X(\nu)$. This function can be seen as a homomorphism between semirings which endows a topology $\tau$ on $\Val_X$, i.e., $\tau$ is the weakest topology on $\Val_X$ such that the evaluation map $\nu \to \nu(\mf{a})$ is continuous for all nonzero ideals $\mf{a}$ on $X$. See \cite[Lemma 4.1]{JM} for another equivalent characterization of the topology $\tau$ on $\Val_X$.

Next we recall the notion of \emph{quasi-monomial valuations}.
Suppose that $f\colon Y\to X$ is a projective birational morphism from a smooth projective variety $Y$ and $E\coloneqq \sum_{i\in I}E_i$ is a simple normal crossing divisor on $Y$. Such a pair $(Y,E)$ is called a \emph{log-smooth pair} over $X$ if $f$ is an isomorphism outside of $\Supp E$. Let $V$ be an irreducible component of $\bigcap_{i\in I'}E_i(\neq\varnothing)$ for some subset $I'\subseteq I$, and $\eta$ the generic point of $V$. Let $z_i$ be an equation in $\mc O_{Y,\eta}$ which locally defines $E_i$ at $\eta$ for $i\in I'$. Thus any $f\in\mathcal{O}_{Y,\eta}$ has the expression $f=\sum\limits_{\beta\in \Z^r_{\ge 0}} a_{\beta} z^{\beta}$ for some $a_{\beta}\in\C(\eta)$ which is either zero or a unit. Here, we use the notation $z^{\beta}=z_1^{\beta_1}z_2^{\beta_2}\cdots z_r^{\beta_r}$ for  $\beta=(\beta_1,\cdots,\beta_r)\in\Z^r_{\ge 0}$.
For each $\alpha\coloneqq (\alpha_1,\cdots,\alpha_r)\in \R^r_{\ge 0}$ and each choice of the point $\eta$, we can define a valuation $\nu_{\alpha,\eta}$ on $\widehat{\mathcal{O}_{Y,\eta}}$ as
\begin{align*}
\nu_{\alpha,\eta}\left(f\right)\coloneqq \min\left\{\sum^r_{i=1}\alpha_i\beta_i\biggm|a_{\beta}\ne 0\right\}.
\end{align*}
\begin{definition}
Any valuation in $\Val_X$ of the form $\nu_{\alpha,\eta}$ as above for some log-smooth pair $(Y,E)$ over $X$ is called a \emph{quasi-monomial valuation} (at $\eta$ with respect to $\alpha\in\R^n_{\geq0}$).
\end{definition}

We denote by $\QM_{\eta}(Y,E)$ the collection of all the quasi-monomial valuations $\nu_{\alpha,\eta}$ defined by the log-smooth pair $(Y,E)$ and $\QM(Y,E)$ denotes the union of all such $\QM_{\eta}(Y,E)$.

We consider $\mathrm{QM}(Y,E)$ as the topological space equipped with the subspace topology induced by $\tau$ on $\Val_X$. Then there exists a \emph{retraction map}
$$r_{(Y,E)}\colon \Val_X\to \QM(Y,E)$$
which is continuous and is the identity on $\QM(Y,E)$ (\cite[Section 4.3]{JM}). This map $r_{(Y,E)}$ allows us to extend the notion of the log discrepancy function $A_{X,\Delta}$ for arbitrary valuations in $\Val_X$ as follows. Let $(Y,E=\sum_{i\in I}E_i)$ be a log-smooth pair of $X$. Let $\nu=\nu_{\alpha,\eta}$ be a quasi-monomial valuation in $\QM(Y,E)$ associated to a generic point $\eta$ of some irreducible component of $\cap_{j\in I'}E_j$ with $I'\subseteq I$ such that $|I'|=r$ and some $\alpha=(\alpha_{1},\dots,\alpha_{r})\in \R_{\ge0}^{r}$. Then the log discrepancy $A_{X,\Delta}(\nu)$ of the quasi-monomial valuation $\nu$ with respect to $(X,\Delta)$ is defined as
$$A_{X,\Delta}(\nu)\coloneqq \sum_{j\in I'}\alpha_{j}A_{X,\Delta}(E_j).$$
We further define the log discrepancy $A_{X,\Delta}(\nu)$ for an arbitrary valuation $\nu\in\mathrm{Val}_X$ with respect to $(X,\Delta)$ as
$$A_{X,\Delta}(\nu)\coloneqq \sup_{(Y,E)}A_{X,\Delta}(r_{(Y,E)}(\nu)),$$
where the sup is taken over all log-smooth pairs $(Y,E)$ over $X$.

We remark that if there is a real number $c>0$ such that $c\alpha\in\Z^r_{\geq0}$, then $\nu_{\alpha,\eta}$ is nothing but a divisorial valuation. 
Note also that the set of divisorial valuations in $\QM(Y,E)$ forms a dense subset of $\QM(Y,E)$ where $\QM(Y,E)$ is considered as a subspace topology of $\Val_X$ induced by $\tau$ (\cite[Remark 3.9]{JM}).

\medskip

Let $\Phi\subseteq \Z_{\ge 0}$ be a semigroup.  The set of ideal sheaves $\mathfrak{a}_{\bullet}:=\{\mathfrak{a}_m\}_{m\in \Phi}$ of $\mathcal{O}_X$ is called a \emph{graded sequence of ideals} of $\mathcal{O}_X$ if for every $m,n\in \Phi$,
$$ \mathfrak{a}_m \cdot \mathfrak{a}_n\subseteq \mathfrak{a}_{m+n}.$$
For a graded sequence of ideals $\mf{a}_\bullet=\{\mf{a}_m\}_{m\in\Phi}$, we define
$$\nu(\mf{a}_\bullet):=\inf\limits_{\substack{m\geq 1 \\ m\in \Phi}}\frac{\nu(\mf{a}_m)}{m}.$$
For an effective $\Q$-Cartier $\Q$-divisor $D$, we define
$$\sigma_{\nu}(D):=\lim\limits_{\substack{\varepsilon\to 0+ \\ \varepsilon \text{ is rational}}}\nu(\mf{a}_{\bullet}(\varepsilon)),$$
where $\mf{a}_{\bullet}(\varepsilon)$ is the graded sequence of ideals in $\mathcal{O}_X$ such that each $\mf{a}_m(\varepsilon)$ is the base ideal of the linear system $|m(D+\varepsilon A)|$ for an ample Cartier divisor $A$, a rational number $\varepsilon\geq 0$, and a positive integer $m$ such that $m\varepsilon\in \Z_{\ge 0}$ and $mD$ are Cartier. By \cite[Theorem 2]{Jow}, the limit is independent of the choice of $A$. Moreover, if $D$ is a big Cartier divisor, it is known that $\sigma_\nu(D)=\sigma_\nu(\mf{b}_\bullet)$ where $\mf{b}_{\bullet}$ is the graded sequence of ideals such that each $\mf{b}_m$ is the base ideal of $|mD|$.

\begin{definition}
For a pair $(X,\Delta)$ and a graded sequence of ideal sheaves $\mathfrak{a}_\bullet$ on $X$, we define the \textit{log canonical threshold} $\lct(X,\Delta,\mathfrak{a}_\bullet)$ as
$$
\lct(X,\Delta,\mathfrak{a}_\bullet):=\inf_{\nu\in\Val^*_X}\frac{A_{X,\Delta}(\nu)}{\nu(\mathfrak{a}_\bullet)}.
$$
\end{definition}

If $\mathfrak{a}_m$ is the base ideal of $|mD|$ for an effective divisor $D$, then we denote $\lct(X,\Delta,\mathfrak{a}_\bullet)=\lct_{\sigma}(X,\Delta,D)$.
The following theorem proves that the log canonical threshold $\lct_{\sigma}(X,\Delta,D)$ can be computed by a quasi-monomial valuation.

\begin{theorem}[{\cite[Theorem 1.1]{Xu}}]\label{thrm:QM compute lct}
Let $(X,\Delta)$ be a klt pair and $\mathfrak{a}_\bullet$  a graded sequence of ideal sheaves on $X$. Then there exists a quasi-monomial valuation $\omega$ which computes $\lct(X,\Delta,\mathfrak{a}_\bullet)$, i.e.,
$$
\lct(X,\Delta,\mathfrak{a}_\bullet)=\frac{A_{X,\Delta}(\omega)}{\omega(\mathfrak{a}_\bullet)}.
$$
\end{theorem}

It is further expected that all the valuations computing $\lct(X,\Delta,\mathfrak{a}_\bullet)$ are quasi-monomial valuations (\cite[Conjecture B]{JM}). See also \cite{CJKL} for a variation of the result.

\medskip

\section{Asymptotic multiplier ideal sheaf}\label{sect:3}

In this section, we define the multiplier ideal sheaf  $\mc{J}(X,\Delta,\norm{D})$ associated to a triple $(X,\Delta,D)$ and study the properties of the closed set defined by $\mc{J}(X,\Delta,\norm{D})$.

We first recall the definition of multiplier ideal sheaf $\mc{J}(X,\Delta)$ associated to a pair $(X,\Delta)$. Let $f:Y\to X$ be a log resolution of a given pair $(X,\Delta)$. Then the \textit{multiplier ideal sheaf} $\mc{J}(X,\Delta)$ associated to $(X,\Delta)$ is defined as
$$ \mc{J}(X,\Delta):=f_*\mc O_X(K_Y-\floor{f^*(K_X+\Delta)})=f_{\ast}\mc{O}_{Y}(-\floor{\Delta_{Y}})$$
where $\Delta_{Y}$ is a divisor on $Y$ such that $K_{Y}+\Delta_{Y}=f^{\ast}(K_{X}+\Delta)$. Analogously, we define an ideal sheaf associated to a triple $(X,\Delta,D)$ that detects the singularities of the triple.

\begin{definition}\label{def:multiple ideal}
For a pair $(X,\Delta)$ and a log resolution $f:Y\to X$ of $(X,\Delta)$, let $\Delta_Y$ be a divisor on $Y$ such that $K_X+\Delta_Y=f^*(K_X+\Delta)$.
\begin{enumerate}[wide = 0pt, leftmargin = 1.4em]
\item Let $F$ be an effective $\Q$-Cartier $\Q$-divisor on $X$. Then we define
$$
\mc{J}(X,\Delta;F):=f_{\ast}\mc{O}_{Y}(-\floor{\Delta_{Y}+f^*F})
$$
for any log resolution $f:Y\to X$ of $(X,\Delta+F)$.

\item Let $\mathfrak{a}$ be a coherent ideal sheaf on $X$. Then, for a positive real number $c$, the \textit{multiplier ideal sheaf} $\mc{J}(X,\Delta;c\cdot\mathfrak{a})$ is defined as
$$ \mc{J}(X,\Delta;c\cdot\mf{a}):=f_{\ast}\mc{O}_{Y}(-\floor{\Delta_{Y}+cF}), $$
where $f:Y\to X$ is a log resolution of both $(X,\Delta)$ and $\mf{a}$ and $F$ is an effective divisor such that $\mf{a}\cdot \mc{O}_{Y}=\mc{O}_{Y}(-F)$.

\item For a divisor $L$ on $X$, let $V\subseteq |L|$ be a linear subseries and $\mf{b}$ be the base ideal of $V$. Then for any positive real number $c$, we define $\mc{J}(X,\Delta;c\cdot V):=\mc{J}(X,\Delta;c\cdot\mf{b})$.
\end{enumerate}
\end{definition}

The definitions are independent of the choice of the resolution $f$.  Clearly,  $\mc{J}(X,\Delta;F)$ in (1), $\mc{J}(X,\Delta;c\cdot\mathfrak{a})$ in (2) and $\mc{J}(X,\Delta;c\cdot|V|)$ in (3) are ideal sheaves in $\mc{O}_X$.
Note that in (1), $\mc{J}(X,\Delta;F)=\mc{J}(X,\Delta+F)$ holds.

\medskip

Let $\mf{b}_{\bullet}=\{\mf{b}_m\}_{m\in\mathbb Z_{\geq0}}$ be a graded sequence of ideals of $\mc{O}_X$, that is, $\mf{b}_m\cdot\mf{b}_n\subseteq \mf{b}_{m+n}$ for any $m,n\in\mathbb Z_{\geq 0}$.
It is easy to check that for any integers $m,l>0$, we have $\mc{J}(X,\Delta;\frac{c}{m}\cdot\mf{b}_m)\subseteq \mc{J}(X,\Delta;\frac{c}{ml}\cdot\mf{b}_{ml})$ for any real number $c>0$.

\begin{lemma}\label{lem:maximal ideal}
Let $(X,\Delta)$ be a pair and $\mf{b}_\bullet$ be a graded ideal sheaf on $X$.
Then for any positive real number $c$, there exists a unique maximal ideal in $\mc{I}=\{\mc{J}(X,\Delta;\frac{c}{m}\cdot\mf{b}_{m})|m\in\mathbb{Z}_{>0}\}$.
\end{lemma}

\begin{proof}
The argument for the proof is basically the same as the proof of \cite[Lemma 11.1.1]{laza2}. For positive integers $p,q$, let $f:Y\to X$ be a log resolution of $(X,\Delta)$, $\mf{b}_{p}$, and $\mf{b}_{pq}$. Write
\begin{align*}
K_{Y}+\Delta_{Y}&=f^{\ast}(K_{X}+\Delta),\\
\mf{b}_{p}\cdot \mc{O}_{Y}&=\mc{O}_{Y}(-F_{p}),\\
\mf{b}_{pq}\cdot \mc{O}_{Y}&=\mc{O}_{Y}(-F_{pq}),
\end{align*}
where $F_{p}$ and $F_{pq}$ are effective divisors with simple normal crossing supports. Since $\mf{b}_{p}^{q}\subseteq \mf{b}_{pq}$,  we have $qF_{p}\ge F_{pq}$. Thus,
$$ \mc{O}_{Y}(-\floor{\Delta_{Y}+\tfrac{c}{p}F_{p}})\subseteq \mc{O}_{Y}(-\floor{\Delta_{Y}+\tfrac{c}{pq}F_{pq}}) $$
which implies $\mc{J}(X,\Delta;\frac{c}{p}\cdot\mf{b}_{p})\subseteq \mc{J}(X,\Delta;\frac{c}{pq}\cdot\mf{b}_{pq})$. This shows that the given family $\mc{I}$ of ideal sheaves has the unique maximal element.
\end{proof}

Lemma \ref{lem:maximal ideal} allows us to define the following.

\begin{definition}\label{def:asympt ideal sheaf}
Let $(X,\Delta)$, $\mf{b}_\bullet$ and $c$ be as above.
\begin{enumerate}[wide = 0pt, leftmargin = 1.4em]
    \item  The maximal element in the family $\mc{I}=\{\mc{J}(X,\Delta;\frac{c}{m}\cdot\mf{b}_{m})|m\in\mathbb Z_{>0}\}$ of multiplier ideal sheaves is called \textit{the asymptotic multiplier ideal sheaf of $(X,\Delta)$ associated to $c$ and $\mf{b}_\bullet$} and is denoted by $\mc{J}(X,\Delta;c\cdot\mathfrak{b}_{\bullet})$.

    \item Let $D$ be an effective $\Q$-Cartier $\Q$-divisor on $X$ and $\mathfrak{b}_{m}$  the base ideal of $|mD|$. Then we call the maximal element $\mc J(X,\Delta;c\cdot\frak{b}_\bullet)$ the \emph{asymptotic multiplier ideal sheaf of $(X,\Delta)$ associated to $c$ and $D$} and denote it by $\mc J(X,\Delta,c\cdot\norm{D})$. If $c=1$, then we call $\mc{J}(X,\Delta,\norm{D}):=\mc{J}(X,\Delta,1\cdot\norm{D})$ the \emph{asymptotic multiplier ideal sheaf associated to the triple $(X,\Delta,D)$}.

\end{enumerate}
\end{definition}

\medskip

\begin{proposition}[\protect{\cite[Lemma 5.4]{ELSV}}] \label{prop:perturb asymp_mult_ideals}
Let $(X,\Delta,D)$ be a triple with a big $\Q$-divisor $D$ on $X$. Then, for any sufficiently small $\varepsilon>0$, we have
$$ \mc{J}(X,\Delta;\norm{D})=\mc{J}(X,\Delta;\norm{(1+\varepsilon)D}). $$
\end{proposition}

\begin{proof}
Let $m$ be some sufficiently large positive integer such that $\mc{J}(X,\Delta;\norm{D})=\mc{J}(X,\Delta;\frac{1}{m}\mathfrak{b}_{m})$, where $\mathfrak{b}_{m}$ denotes the base ideal of $|mD|$. Then, for sufficiently small $\varepsilon\geq 0$, we have
\begin{align*}
\mc{J}(X,\Delta;\norm{D})&=\mc{J}(X,\Delta;\tfrac{1}{m}\cdot\mf{b}_{m})\\
&=\mc{J}(X,\Delta;\tfrac{1+\varepsilon}{m}\cdot\mf{b}_{m})\\
&\subseteq \mc{J}(X,\Delta;\norm{(1+\varepsilon)D})\\
&\subseteq \mc{J}(X,\Delta;\norm{D}). \qedhere
\end{align*}
\end{proof}

\medskip

The following results describe the locus $\pnklt(X,\Delta,D)$.

\begin{proposition}\label{prop:center in pnklt}
Let $(X,\Delta,D)$ be a triple with an effective $\Q$-divisor $D$ on $X$. Then  we have
$$\pnklt(X,\Delta,D)\subseteq\mc{Z}(\mc{J}(X,\Delta,\norm{D})).$$
In particular, if $\mc{J}(X,\Delta,\norm{D})=\mc{O}_X$, then $(X,\Delta,D)$ is weakly pklt.
\end{proposition}

\begin{proof}
For a positive integer $m$, let $f:Y\to X$ be a log resolution of $(X,\Delta)$ such that there exists a decomposition $f^{\ast}|mD|=|M_{m}|+F_{m}$ into free part $|M_{m}|$ and fixed part $F_{m}$ having simple normal crossing support. Write $K_{Y}+\Delta_{Y}=f^{\ast}(K_{X}+\Delta)$ for some divisor $\Delta_{Y}$ on $Y$. Suppose that $E$ is a prime divisor over $X$ satisfying $a(E;X,\Delta,D)=A_{X,\Delta}(E)-\sigma_{E}(D)\le 0$. By taking $f$ higher if necessary, we may assume that $E$ is a prime divisor on $Y$. Let $f^{\ast}D=P+N$ be the divisorial Zariski decomposition. Then since $mN\le F_{m}$ holds, we have $1\le \mult_{E}(\Delta_{Y}+N)\le \mult_{E}(\Delta_{Y}+\frac{1}{m}F_{m})$. Thus $\mult_E(-\floor{\Delta_Y+\frac{1}{m}F_m})\leq -1$. Since we may assume that $m$ is an integer such that $\mc{J}(X,\Delta,\norm{D})=f_*\mc O_Y(-\floor{\Delta_Y+\frac{1}{m}F_m})$, the center $c_X(E)$ is contained in $\mc{Z}(\mc{J}(X,\Delta,\norm{D}))$.
\end{proof}

\begin{proposition}\label{prop:pnklt(1-e) in Cent}
Let $(X,\Delta,D)$ be a triple with a big $\Q$-divisor $D$ on $X$. Assume that $(X,\Delta)$ is klt. Then for any $0<\varepsilon\le 1$, we have
$$ \mc{Z}(\mc{J}(X,\Delta,(1-\varepsilon)\cdot\norm{D}))\subseteq \pnklt(X,\Delta,D).$$
\end{proposition}

\begin{proof}
By \cite[Corollary 11.2.22]{laza2}, there exists a log resolution $f:Y\to X$, a positive integer $m_{0}$, and an effective divisor $G$ on $Y$ satisfying the following conditions:
\begin{enumerate}[wide = 0pt, leftmargin = 1.4em]
\item $\Exc(f)\cup\Supp f_{\ast}^{-1}\Delta$ is a reduced simple normal crossing divisor, and
\item for any $m\ge m_{0}$
$$ \mf{b}(f^{\ast}|kmD|)\otimes \mc{O}_{Y}(-kG)\subseteq \mf{b}(f^{\ast}|mD|)^{k} \;\text{ holds for any }  k\ge 1. $$
\end{enumerate}

Let $\phi:Y'\to Y$ be a birational morphism such that $g=f\circ \phi$ is a resolution of $\mf{b}(|mD|)$. We can write $K_{Y'}+\Delta_{Y'}=g^{\ast}(K_{X}+\Delta)$ for some divisor $\Delta_{Y'}$ on $Y'$  and $g^{\ast}|mD|=|M_{m}|+F_{m}$ where $|M_{m}|$ is free and $F_{m}$ is an effective divisor with simple normal crossing support. If we choose $m$ sufficiently large and divisible, then
$$ g_{\ast}\mc{O}_{Y'}(-\floor{\Delta_{Y'}+\tfrac{1-\varepsilon}{m}F_{m}})=\mc{J}(X,\Delta,(1-\varepsilon)\cdot\norm{D}). $$
Let $\Delta_Y$ be the divisor on $Y$ such that $K_Y+\Delta_Y=g^*(K_X+\Delta)$ and define the log canonical threshold $\mu=\lct(Y,\Delta_{Y};G):=\sup\{t\geq 0\;|(Y,\Delta_Y+tG)\;\text{is sub-lc}\}$. Note that $\mu>0$ holds since $(X,\Delta)$ is klt. Let $m$ be a sufficiently large integer such that $\frac{1}{\mu m}<\varepsilon$.

Suppose that $Z$ is an irreducible component of $\mc{Z}(\mc{J}(X,\Delta,(1-\varepsilon)\cdot\norm{D})$. Then we can find a prime divisor $E$ on $Y$ such that $c_{X}(E)=Z$ and
$$ \mult_{E}(\Delta_{Y'}+\tfrac{1-\varepsilon}{m}F_{m})\ge 1. $$
Note that the last inequality is equivalent to
$$ A_{X,\Delta}(E)-\frac{1-\varepsilon}{m}\mult_{E}F_{m}\le 0 $$
and we have
$$ \frac{1}{m}\mult_{E}\phi^{\ast}G\le \frac{1}{\mu m}A_{Y,\Delta_Y}(E)=\frac{1}{\mu m}A_{X,\Delta}(E)\le \varepsilon A_{X,\Delta}(E). $$

Let $\psi:Y''\to Y'$ be a birational morphism such that $h=g\circ\psi$ is a log resolution of $\mf{b}(|kmD|)$. We can write $K_{Y''}+\Delta_{Y''}=h^{\ast}(K_{X}+\Delta)$ for some divisor $\Delta_{Y''}$ on $Y''$ and $h^{\ast}|kmD|=|M_{km}|+F_{km}$ where $|M_{km}|$ is free and $F_{km}$ is an effective divisor with simple normal crossing support. By construction, we have
$$ k\psi^{\ast}F_{m}\le k\psi^{\ast}\phi^{\ast}G+F_{km}. $$
Therefore, we obtain that
\begin{align*}
0&\ge A_{X,\Delta}(E)-\frac{1-\varepsilon}{m}\mult_{E}F_{m}\\
&\ge A_{X,\Delta}(E)-\frac{1-\varepsilon}{m}\mult_{E}\phi^{\ast}G-\frac{1-\varepsilon}{km}\mult_{E}F_{km}\\
&\ge (1-\varepsilon(1-\varepsilon))A_{X,\Delta}(E)-\frac{1-\varepsilon}{km}\mult_{E}F_{km}.
\end{align*}
By letting $k\to \infty$, we finally have
$$ (1-\varepsilon(1-\varepsilon))A_{X,\Delta}(E)\le (1-\varepsilon)\sigma_{E}(D). $$
If $Z\not\subseteq \pnklt(X,\Delta,D)$, then $A_{X,\Delta}(E)>\sigma_{E}(D)$. Thus, we can conclude that
$$ (1-\varepsilon(1-\varepsilon))A_{X,\Delta}(E)<(1-\varepsilon)A_{X,\Delta}(E). $$
Since $(X,\Delta)$ is klt, $A_{X,\Delta}(E)>0$. Thus we have
$$ (1-\varepsilon(1-\varepsilon))<(1-\varepsilon) $$
which is a contradiction.
\end{proof}

In summary, we have the following inclusions: for $0<\varepsilon\leq 1$
\begin{equation*}
 \mc{Z}(\mc{J}(X,\Delta,(1-\varepsilon)\cdot\norm{D})\subseteq \pnklt(X,\Delta,D)\subseteq \mc{Z}(\mc{J}(X,\Delta,\norm{D})).
\end{equation*}
We will prove the cases where we can let $\varepsilon=0$ in Theorems \ref{thrm:J_pnklt eff -K-Delta} and \ref{thrm: fg}.

\medskip

\begin{theorem} \label{thrm:pklt<->J=0}
Let $(X,\Delta,D)$ be a triple with a big $\Q$-divisor $D$ on $X$. Then $(X,\Delta,D)$ is pklt if and only if $\mc{J}(X,\Delta,\norm{D})=\mc O_X$.
\end{theorem}

\begin{proof}
If $(X,\Delta,D)$ is pklt, then $(X,\Delta)$ is a klt pair. Thus by Theorem \ref{thrm:QM compute lct}, there exists a quasi-monomial valuation $\nu_0$ computing the log canonical threshold $\lct_{\sigma}(X,\Delta;D)$, that is,
$$
\inf_{\nu\in\Val^*_X}\frac{A_{X,\Delta}(\nu_0)}{\sigma_\nu(D)}=\frac{A_{X,\Delta}(\nu_0)}{\sigma_{\nu_0}(D)}.
$$
Let $(Y,E)$ be a log-smooth model over $(X,\Delta)$ such that $\nu_0\in \QM_{\eta}(Y,E)$ where $E=\sum_{i=1}^r E_{i}$ is a reduced simple normal crossing divisor and $\eta=c_{Y}(\nu_0)$ the generic point of an irreducible  component of $E_{1}\cap \cdots \cap E_{r}$.

Let $\phi(\nu)=A_{X,\Delta}(\nu)-\sigma_{\nu}(D)$ for $\nu\in\QM_{\eta}(Y,E)$. Then clearly $\phi$ is homogeneous of degree 1. Furthermore, note that $A_{X,\Delta}(\nu)$ is linear on $\QM_{\eta}(Y,E)$ and $\sigma_{\nu}(D)$ is concave on $\QM_{\eta}(Y,E)$. Therefore, $\phi$ is convex on $\QM_{\eta}(Y,E)$. Since every convex function is locally Lipschitz, we can find positive real numbers $C,\delta>0$ such that
$$ |\phi(\nu_0)-\phi(w)|<C|\nu_0-w| $$
holds for all valuations $w\in \QM_{\eta}(Y,E)$ with $|\nu_0-w|\le \delta$. Since the divisorial valuations in $\QM_{\eta}(Y,E)$ are dense in $\QM_\eta(Y,E)$, for any $t>0$ there always exists some divisorial valuation $\omega$ such that $|\omega-\nu_0|<t$.

On the other hand, by Lemma \ref{lem:diopahntine}, for each $t>0$, there exists a divisorial valuation $w_{t}\in \QM_{\eta}(Y,E)$ and a positive rational number $q$ such that
\begin{enumerate}[label=$\bullet$]
\item $qw_{t}=\ord_{F}$, where $F$ is a prime divisor over $X$,
\item $|\nu-w_{t}|<\frac{t}{q}$.
\end{enumerate}
Since $(X,\Delta,D)$ is pklt, there is a positive real number $\beta>0$ such that $A_{X,\Delta}(E)-\sigma_{E}(D)\ge \beta$ for all prime divisors $E$ over $X$. By taking a sufficiently small $t>0$, we may assume
$$ \phi(q\nu)\ge \phi(qw_{t})-|\phi(q\nu)-\phi(qw_{t})|>\beta-Ct>0. $$
Hence, $\varphi(\nu)=A_{X,\Delta}(\nu)-\sigma_{\nu}(D)>0$. It follows that
$$\lct_{\sigma}(X,\Delta,D)=\frac{A_{X,\Delta}(\nu_0)}{\sigma_{\nu_0}(D)}>1.$$
Now, the assertion $\mc{J}(X,\Delta,\norm{D})=\mc{O}_X$ follows from \cite[Lemma 1.56]{Xu2}.

Conversely, assume that $\mc{J}(X,\Delta,\norm{D})=\mc{O}_{X}$. Then, by Proposition \ref{prop:perturb asymp_mult_ideals},  for a sufficiently small $\varepsilon>0$,
we have
$$ \mc{J}(X,\Delta,\norm{(1+\varepsilon)D})=\mc{J}(X,\Delta,\norm{D})=\mc{O}_{X}. $$
By Proposition \ref{prop:center in pnklt}, the triple $(X,\Delta,(1+\varepsilon)D)$ is weakly pklt. Thus, by Lemma \ref{lem:weakly pklt}, $(X,\Delta,D)$ is pklt.
\end{proof}

\begin{lemma}[{cf. \cite[Lemma 2.7]{LX}}]\label{lem:diopahntine}
Let $\alpha=(\alpha_1,\cdots,\alpha_r)\in\mathbb R^r_{\geq 0}$ and fix a real number $\varepsilon>0$. Then there exist $\alpha'\in\Q^r_{\geq 0}$ and a positive integer $n$ such that
\begin{itemize}
    \item $n\alpha'\in\mathbb{Z}^r_{\geq 0}$, and
    \item $\norm{\alpha-\alpha'}<\frac{\varepsilon}{n}$.
\end{itemize}
\end{lemma}

\begin{proof}
By applying Weyl's criterion (cf. \cite[Theorem 1.6.1]{KN}), we obtain that there exists a positive integer $n$ such that
$$
n\alpha_i-\floor{n\alpha_i}<\frac{\varepsilon}{\sqrt{r}} \;\text{for each }i.
$$
Let $\alpha':=(\alpha'_1,\cdots,\alpha'_r)$ where $\alpha'_i=\frac{\floor{n\alpha_i}}{n}$.
Then we have $\norm{n\alpha-n\alpha'}<\varepsilon$, which gives us the second condition $\norm{\alpha-\alpha'}<\frac{\varepsilon}{n}$.
\end{proof}

\medskip

Now we are ready to prove the first main theorem of the paper, Theorem \ref{maintheorem1}.

\begin{proof}[Proof of Theorem \ref{maintheorem1}]
The equivalence $(1)\Leftrightarrow (2)$  follows from Theorem \ref{thrm:pklt<->J=0}.

The implication $(2)\Rightarrow (3)$ follows from the construction of $\mc{J}(X,\Delta,\norm{D})$.  Indeed, there is a general effective $\Q$-divisor $D'\sim_{\Q}D$ such that $\mc{J}(X,\Delta,\norm{D})=\mc{J}(X,\Delta+D')$ (cf. \cite[Proposition 9.2.26]{laza2}). By (2), $\mc{J}(X,\Delta+D')=\mc{O}_{X}$ and the pair $(X,\Delta+D')$ is klt.

The implication $(3)\Rightarrow (1)$ can be proved as follows.
Let $D'$ be a klt complement of $(X,\Delta,D)$. Since $(X,\Delta+D')$ is klt, there exists  $\varepsilon>0$ such that $0<\varepsilon< A_{X,\Delta+D'}(E)$ holds for any prime divisor $E$ over $X$. One can check that $A_{X,\Delta+D'}(E)\leq a(E;X,\Delta,D)$. Therefore, $\inf_E a(E;X,\Delta,D)\geq \varepsilon>0$ and this shows that $(X,\Delta,D)$ is pklt.
\end{proof}

\begin{remark}
Suppose that $(X,\Delta)$ is a klt pair. Then the triple $(X,\Delta,D)$ is weakly pklt for any nef divisor $D$. Thus, by Theorem \ref{maintheorem1} if $D$ is nef and big, then $\mc{J}(X,\Delta,\norm{D})=\mc{J}(X,\Delta)$. Note also that if $\Delta=0$, then $\mc{J}(X,\Delta,\norm{D})=\mc{J}(X,\norm{D})$. Therefore, the ideal sheaf $\mc{J}(X,\Delta,\norm{D})$ can be considered as a generalization of both the multiplier ideal sheaf $\mc{J}(X,\Delta)$ and the asymptotic multiplier ideal sheaf $\mc{J}(X,\norm{D})$.
\end{remark}

A \emph{Fano type} variety $X$ is a $\Q$-factorial normal projective variety such that $(X,\Delta)$ is klt and $-(K_X+\Delta)$ is ample for some effective divisor $\Delta$ on $X$ (\cite{PS}). It is known that Fano type varieties are Mori dream spaces by \cite{bchm}. For a pair $(X,\Delta)$ with pseudoeffective $-(K_X+\Delta)$, we will say that $(X,\Delta)$ is \emph{pklt} if the triple $(X,\Delta,-(K_X+\Delta))$ is pklt.

\begin{corollary}\label{cor: FT characterization}
Let $X$ be a $\Q$-factorial normal projective variety.
Then $X$ is a Fano type variety if and only if there exists an effective $\Q$-divisor $\Delta$ on $X$ such that $-(K_X+\Delta)$ is big and $(X,\Delta)$ is pklt.
\end{corollary}

\begin{proof}
If $X$ is a Fano type variety, then by definition there exists an effective divisor $\Delta$ on $X$ such that $(X,\Delta)$ is klt and $-(K_X+\Delta)$ is ample. Then $(X,\Delta,-(K_X+\Delta))$ is clearly a pklt triple. Conversely, if $\Delta$ is an effective $\Q$-divisor such that $(X,\Delta,-(K_X+\Delta))$ is a pklt triple and $-(K_X+\Delta)$ is big, then by Theorem \ref{maintheorem1} there exists an effective $\Q$-divisor $D'\sim_{\Q}-(K_X+\Delta)$ such that $(X,\Delta+D')$ is klt. Since $-K_X$ is big and $K_X+\Delta+D'\sim_{\Q}0$, it is easy to see that $X$ is of Fano type.
\end{proof}

\begin{remark}\label{remk:CP 4.4}
This replaces the proof of \cite[Theorem 5.1]{CP}. The proof of \cite[Theorem 5.1]{CP} depends on \cite[Proposition 4.4]{CP} whose proof contains a gap. The equality in \cite[Proposition 4.4]{CP} is proved in Theorems \ref{thrm:J_pnklt eff -K-Delta} and \ref{thrm: fg} with some additional conditions.
\end{remark}

\bigskip

Next we present an alternative and more geometric construction of the asymptotic ideal sheaf $\mc{J}(X,\Delta,\norm{D})$ associated to the triple $(X,\Delta,D)$.

Let $f:Y\to X$ be a log resolution of $(X,\Delta)$. Then we can write $K_{Y}+\Delta_{Y}=f^{\ast}(K_{X}+\Delta)$ for some divisor $\Delta_Y$ on $Y$. Denote by $N$ the negative part of the divisorial Zariski decomposition of $f^{\ast}D$. We define an ideal sheaf $\mc{J}(X,\Delta,D;f)\subseteq\mc{O}_X$ as
$$\mc{J}(X,\Delta,D;f)=f_{\ast}\mc{O}_{Y}(-\floor{\Delta_{Y}+N}). $$
Note that for a prime divisor $E$ on $Y$,
\begin{equation}\label{**}\tag{$**$}
\mult_E(\floor{\Delta_Y+N})\geq 1 \text{ if and only if } a(E;X,\Delta,D)\le 0.
\end{equation}
Note also that if $D$ is nef, then $\mc{J}(X,\Delta,D;f)$ coincides with the multiplier ideal sheaf $\mc{J}(X,\Delta)$. Since the multiplier ideal sheaf $\mc{J}(X,\Delta)$ is independent of $f$, so is $\mc{J}(X,\Delta,D;f)$. However, as we can deduce from the following monotone property, this is no longer the case if $D$ is not nef.

\begin{proposition}\label{prop:decreasing ideal}
Suppose that we have two log resolutions $f:Y\to X$, $g:W\to X$ of $(X,\Delta)$ such that $g$ dominates $f$. In other words, there exists a surjective birational morphism $h:W\to Y$ such that $g=f\circ h$. Then we have
$$\mc{J}(X,\Delta,D;g) \subseteq \mc{J}(X,\Delta,D;f). $$
\end{proposition}
\begin{proof}
For some divisors $\Delta_W$ on $W$ and $\Delta_Y$ on $Y$, we have $K_Y+\Delta_Y=f^*(K_X+\Delta)$ and $K_W+\Delta_W=g^*(K_X+\Delta)$.
Let $f^*D=P+N$, $g^*D=P'+N'$ be the divisorial Zariski decompositions.
Then we have $h^*N\leq N'$ by Lemma \ref{lemma-np-behavior}. Thus $-\floor{\Delta_W+N'}\leq-\floor{\Delta_W+h^*N}$ and the following holds:
$$
h_*\mc O_W(-\floor{\Delta_W+N'})\subseteq h_*\mc O_W(-\floor{\Delta_W+h^*N}).
$$
Therefore, we have
\begin{align*}
\mc{J}(X,\Delta,D;g)&=g_*\mc O_W(-\floor{\Delta_W+N'})\\
&=f_*h_*\mc O_W(-\floor{\Delta_W+N'})\\
&\subseteq f_*\mc O_Y(-\floor{\Delta_Y+N})=\mc{J}(X,\Delta,D;f).\qedhere
\end{align*}

\end{proof}

It is easy to see that the strict inclusion $\mc{J}(X,\Delta,D;g) \subsetneq \mc{J}(X,\Delta,D;f)$ holds in general. Suppose that $(X,\Delta)$ is a klt pair such that $D$ is movable. If there exists a birational morphism $g:W\to X$ and a prime exceptional divisor $E$ on $W$ such that   $\sigma_E(D)\geq 1$, then
$$\mc{J}(X,\Delta,D;g)=g_*\mc O_Y(-\floor{\Delta_Y+N'})\subsetneq \mc{O}_X(-\floor{\Delta})=\mc{J}(X,\Delta,D;\id_{X}).$$

\bigskip

For a given triple $(X,\Delta,D)$, let us consider the family of ideal sheaves $$\mf{J}(X,\Delta,D):=\left\{\mc{J}(X,\Delta,D;f)| \text{ $f:Y\to X$ is a log resolution of $(X,\Delta)$}\right\}.$$
By Proposition \ref{prop:decreasing ideal}, the family is partially ordered by the inclusion relation. It is also easy to see that if there exists a minimal element in the family $\mf{J}(X,\Delta,D)$, then it is unique.

\begin{definition}\label{def:pnklt ideal sheaf}
Let $(X,\Delta,D)$ be a triple with a pseudoeffective divisor $D$ on $X$.
The minimal element in the family $\mf{J}(X,\Delta,D)$ of ideal sheaves is called the \emph{potentially non-klt (pnklt) ideal sheaf} of $(X,\Delta,D)$ and we denote it by $\mc{J}_{\pnklt}(X,\Delta,D)$ (provided that the minimal element exists in $\mf{J}(X,\Delta,D)$).
\end{definition}

By construction, there exist a sufficiently large positive integer $m$ and a log resolution $f:Y\to X$ as above such that $\mc{J}(X,\Delta,\norm{D})=f_*\mc O_Y(-\floor{\Delta_Y+\frac{1}{m}F_m})$ where $F_m$ is the fixed part of the linear system $f^*|mD|$ having simple normal crossing support. Since $N\leq \frac{1}{m}F_m$ where $N$ is the negative part of the divisorial Zariski decomposition of $f^*D$, we have
\begin{center}
$f_*\mc O_Y(-\floor{\Delta_Y+\frac{1}{m}F_m})\subseteq f_*\mc O_Y(-\floor{\Delta_Y+N})$
\end{center}
and this implies
\begin{equation}\tag{$\#$}\label{sharp}
\mc {J}_{\pnklt}(X,\Delta,D)\supseteq\mc{J}(X,\Delta,\norm{D})
\end{equation}
if $\mc{J}_{\pnklt}(X,\Delta,D)$ exists.

\medskip

We can also easily observe that
\begin{enumerate}[(a)]
    \item $\pnklt(X,\Delta,D)=\bigcup_{\mathcal{J}\in \mathfrak{J}(X,\Delta,D)} \mc{Z}(\mc{J})$, where $\pnklt(X,\Delta,D)$ is the union of the centerd $c_X(E)$ of prime divisors $E$ over $X$ such that $a(E,X,\Delta,D)\leq 0$ (Definition \ref{def:pnklt locus}), and
    \item if $\mc{J}_{\pnklt}(X,\Delta,D)$ exists in $\mf{J}(X,\Delta,D)$, then we see that 
    $$\pnklt(X,\Delta,D)=\mc{Z}(\mc{J}_{\pnklt}(X,\Delta,D)).$$
\end{enumerate}

\medskip

The existence of the pnklt ideal sheaf $\mc{J}_{\pnklt}(X,\Delta,D)$ is not obvious unless $D$ is big and admits a birational Zariski decomposition (Theorem \ref{thrm:J_pnklt eff -K-Delta}), or $D$ is effective and is finitely generated (Theorem \ref{thrm: fg}). Here, for a normal variety $X$ and a $\Q$-Cartier $\Q$-divisor $D$ on $X$, we say $D$ is \emph{finitely generated} if $\bigoplus_{m\ge 0}H^0(X,\mathcal{O}_X(\floor{mD}))$ is a finitely generated $\C$-algebra.

\begin{theorem}\label{thrm:J_pnklt eff -K-Delta}
Let $(X,\Delta,D)$ be a triple with a big $\Q$-divisor $D$ on $X$.
Suppose that $D$ admits a birational Zariski decomposition.
Then $\mf{J}(X,\Delta,D)$ has the minimal element $\mc{J}_{\pnklt}(X,\Delta,D)$. Furthermore, we have
$$\mathcal{J}_{\pnklt}(X,\Delta,D)=\mathcal{J}(X,\Delta,\norm{D})\;\;\text{and}$$
$$ \pnklt(X,\Delta,D)=\mc{Z}(\mc{J}_{\pnklt}(X,\Delta,D)).$$
\end{theorem}

\begin{proof}
By assumption, there exists a log resolution $f:Y\to X$ of $(X,\Delta)$ such that the positive part $P=\pp(f^{\ast}D)$ is nef and $\Exc(f)\cup\Supp f_{\ast}^{-1}\Delta\cup \Supp N$ where $N=\np(f^{\ast}D)$ is the negative part is a reduced simple normal crossing divisor. Suppose that $h:W\to Y$ is a projective birational morphism such that $g=f\circ h:W\to X$ is a log resolution of $(X,\Delta)$. Then by Lemma \ref{lemma-np-behavior}, we have
$$ N'=\np(g^{\ast}D)=h^{\ast}N. $$
Let $\Delta_{W}$ and $\Delta_{Y}$ be divisors as in Proposition \ref{prop:decreasing ideal}. Applying \cite[Lemma 9.2.19]{laza2} with $D=\Delta_{Y}+N$, we have
$$ h_{\ast}\mc{O}_{W}(-\floor{\Delta_{W}+N'})=h_{\ast}\mc{O}_{W}(K_{W/Y}-\floor{h^{\ast}(\Delta_{Y}+N)})=\mc{O}_{Y}(-\floor{\Delta_{Y}+N}). $$
Hence, $\mc{J}(X,\Delta,D;f)=\mc{J}(X,\Delta,D;g)$ and this shows that $\mc{J}(X,\Delta,D;f)$ is a  minimal element in $\mf{J}(X,\Delta,D)$.

Suppose that $\mc{J}(X,\Delta,D;f')\in \mf{J}(X,\Delta,D)$ is another minimal element with $f':Y'\to X$. Let $g:W\to X$ be a log resolution of $(X,\Delta)$ which factors through $f$ and $f'$. Then by Proposition \ref{prop:decreasing ideal} and the minimality of $\mc{J}(X,\Delta,D;f)$ and $\mc{J}(X,\Delta,D;f')$, we have
$$\mc{J}(X,\Delta,D;f)=\mc{J}(X,\Delta,D;g)=\mc{J}(X,\Delta,D;f').$$
Therefore, $\mf{J}(X,\Delta,D)$ has the unique minimal element $\mc{J}_{\pnklt}(X,\Delta,D)$ and we have
$$
\mc{J}_{\pnklt}(X,\Delta,D)=\mc{J}(X,\Delta,D;f)=f_*\mc{O}_Y(-\floor{\Delta_Y+N}).
$$
Note that by (\ref{sharp}), we have
$$
\mc{J}_{\pnklt}(X,\Delta,D)\supseteq \mc{J}(X,\Delta,\norm{D}).
$$
Note that for $s\in\mc{O}_X$, $s\in\mc{J}_{\pnklt}(X,\Delta,D)$ if and only if 
$$ \nu(s)>\nu(N)-A_{X,\Delta}(\nu)\text{ for every divisorial valuation $\nu$ of $X$}$$
where $N$ is the negative part of the Zariski decomposition of $f^*D$.
Since $\nu(N)=\sigma_{\nu}(D)$, $s\in\mc{J}_{\pnklt}(X,\Delta,D)$ implies that $\nu(s)>\sigma_{\nu}(D)-A_{X,\Delta}(\nu)$.
Therefore by \cite[Theorem 3.11]{Leh} or \cite[Corollary 4.5]{Kim}\footnote{In \cite{Leh} and \cite{Kim}, $\sigma_{\nu}(D)$ is denoted by $\nu(\|D\|)$}, $s\in\mc{J}(X,\Delta,\norm{D})$ holds and we have $\mc{J}_{\pnklt}(X,\Delta,D)= \mc{J}(X,\Delta,\norm{D})$.
\end{proof}

\begin{theorem}\label{thrm: fg}
Let $(X,\Delta,D)$ be a triple with an effective $\Q$-divisor $D$ on $X$. Suppose that $D$ is finitely generated.
Then $\mf{J}(X,\Delta,D)$ has the minimal element $\mc{J}_{\pnklt}(X,\Delta,D)$. Furthermore, we have
$$\mc{J}_{\pnklt}(X,\Delta,D)=\mc{J}(X,\Delta,\norm{D})\;\;\;\text{and}$$
$$ \pnklt(X,\Delta,D)=Z(\mc{J}_{\pnklt}(X,\Delta,D)).$$
\end{theorem}

\begin{proof}
Since $D$ is finitely generated, by \cite[Proposition 4.7]{ELMNP}, there exists a log resolution $f:Y\to X$ of $(X,\Delta)$ and a positive integer $d$ such that
$$\Mob (df^{\ast}D)=df^{\ast}D-\Fix f^{\ast}|dD|$$
is base point free and $\Fix f^{\ast}|kdD|=k\Fix f^{\ast}|dD|$ for any $k\ge 1$. By definition,
$$ \np(f^{\ast}D)=\tfrac{1}{d}\Fix f^{\ast}|dD|. $$
This implies that
\begin{align*}
f_{\ast}\mc{O}_{Y}(-&\floor{\Delta_{Y}+\np(f^{\ast}D})\\
&=f_{\ast}\mc{O}_{Y}(-\floor{\Delta_{Y}+\tfrac{1}{d}\Fix f^{\ast}|dD|})\\
&=\mc{J}(X,\Delta,\norm{D}).
\end{align*}
Hence, combining with Proposition \ref{prop:decreasing ideal}, the minimal element of $\mathfrak{J}(X,\Delta,D)$ is $\mathcal{J}(X,\Delta,\norm{D})$.
\end{proof}

As an application of Theorem \ref{maintheorem1},  we prove Theorem \ref{maintheorem4}.

\begin{proof}[Proof of Theorem \ref{maintheorem4}]
Let $(X,\Delta,D)$ be a pklt triple and $B$ an ample divisor on $X$ such that $K_{X}+\Delta+D+B$ is nef.  Define $\lambda:=\inf\{t\ge 0 \mid K_X+\Delta+D+tB\text{ is nef}\}$. If $\lambda=0$, then $K_X+\Delta+D$ is nef and there is nothing further to do. Assume that $\lambda>0$. For a fixed rational number $\varepsilon\in (0,\lambda)$, $K_X+\Delta+D+\varepsilon B$ is not nef and by Theorem \ref{maintheorem1}, there is an effective divisor $\Delta_{\varepsilon}\sim_{\Q}D+\varepsilon B$ such that $(X,\Delta+\Delta_{\varepsilon})$ is klt. Thus we can find a $(K_{X}+\Delta+\Delta_{\varepsilon})$-negative extremal ray $R$ in $\overline{\NE}(X)$ such that
$$ (K_{X}+\Delta+\Delta_{\varepsilon}+(\lambda-\varepsilon)B)\cdot R=0.$$
One can easily check that $(K_{X}+\Delta+D)\cdot R<0$ and $(K_{X}+\Delta+D+\lambda B)\cdot R=0$.
By the contraction theorem, there exists an extremal contraction $\varphi=\varphi_R:X\to Y$ for the ray $R$. If $\varphi$ yields a Mori fiber space, then we stop the MMP. If $\varphi$ is birational, then $\varphi$ is either a divisorial contraction or small contraction. If it is a divisorial contraction, then let $(X',\Delta',D')=(Y,\varphi_*\Delta,\varphi_*D)$. If $\varphi$ is a small contraction, then the flip $\varphi^+:X\dashrightarrow X^+/Y$ exists by \cite[Corollary 1.4.1]{bchm} and let $(X',\Delta',D')=(X^+,\varphi^+_*\Delta,\varphi^+_*D)$. In either case,  the map $X\dashrightarrow X'$ is $(K_X+\Delta+D)$-negative and $(X',\Delta',D')$ is pklt by Lemma \ref{lem:D-MMP invariant}. We can start over the process and continue running the MMP.
\end{proof}

The termination of the MMP on $(X,\Delta,D)$ is unknown.

\section{Vanishing theorems}\label{sect:4}

In this last section, we prove several vanishing theorems in regard to potential triples. We first prove the following vanishing theorem.

\begin{theorem} \label{thrm:Kodaira vanishing_semiample}
Let $X$ be a smooth projective variety and $M$ a semiample line bundle on $X$. Then
$$ H^{q}(X,\mc{O}_{X}(K_{X}+M))=0 $$
for all $q>\dim X-\kappa(M)$.
\end{theorem}
\begin{proof}
Since $M$ is semiample, there exists a surjective projective morphism $\mu:X\to Z$ such that $\dim Z=\kappa(M)$, $\mu_{\ast}\mc{O}_{X}=\mc{O}_{Z}$, and $M=\mu^{\ast}A$, where $A$ is an ample $\Q$-divisor on $Z$ (cf. \cite[Theorem 2.1.27]{laza1}). For some sufficiently large integer $m$, we have
$$ H^{q}(Z, R^{p}\mu_{\ast}\mc{O}_{X}(K_{X})\otimes \mc{O}_{Z}(mA))=0 $$
for all $p\ge 0$ and $q>0$. Thus, using the Leray spectral sequence, we obtain
$$ H^{q}(X,\mc{O}_{X}(K_{X}+mM))=H^{0}(Z,R^{q}\mu_{\ast}\mc{O}_{X}(K_{X})\otimes \mc{O}_{Z}(mA)) $$
for all $q>0$. Since $\dim Z=\kappa(M)$, $R^{q}\mu_{\ast}\mc{O}_{X}(K_{X})=0$ holds for all $q>\dim X-\kappa(M)$ (\cite[Theorem 2.1]{Kol1}). Hence,
$$ H^{q}(X,\mc{O}_{X}(K_{X}+mM))=0 $$
for $q>\dim X-\kappa(M)$. Now, the desired result follows from the Koll\'ar's injectivity theorem (\cite[Theorem 2.2]{Kol1}).
\end{proof}

We also obtain the following immediate generalization of Kawamata-Viehweg vanishing theorem.

\begin{theorem}\label{thrm:Kodaira vanishing_semiample2}
Let $X$ be a smooth projective variety and $L$ a line bundle on $X$. Assume that $L\sim_{\Q}\Delta+M$, where $M$ is a semiample divisor and $\Delta=\sum a_{i}\Delta_{i}$ is a $\Q$-divisor with simple normal crossing support and fractional coefficients $0\le a_{i}<1$ for all $i$. Then
$$ H^{q}(X,\mc{O}_{X}(K_{X}+L))=0 $$
for all $q>\dim X-\kappa(M)$.
\end{theorem}
\begin{proof}
Using the cyclic covering trick which is used to prove the Kawamata-Viehwheg vanishing theorem, we can apply Theorem \ref{thrm:Kodaira vanishing_semiample} to obtain the vanishing of cohomology.
\end{proof}

Note that in Theorem \ref{thrm:Kodaira vanishing_semiample2}, if $M$ is additionally big, then the triple $(X,\Delta,M)$ is pklt, or equivalently, $\mc{J}(X,\Delta,\norm{M})=\mc O_X$ by Theorem \ref{maintheorem1}. Theorem \ref{maintheorem3} treats the case where $\mc{J}(X,\Delta,\norm{D})\neq \mc O_X$. It can be considered as a generalization of the Nadel vanishing theorem.

\begin{proof}[Proof of Theorem \ref{maintheorem3}]
By construction, $\mc{J}(X,\Delta,\norm{D})=\mc{J}(X,\Delta;\frac{1}{m}\mf{b}_{m})$ for some sufficiently large positive integer $m$ where $\mf{b}_{m}$ is the base ideal of $|mD|$. For a common log resolution $f:Y\to X$ of both $(X,\Delta)$ and $\mf{b}_{m}$, we can write $K_{Y}+\Delta_{Y}=f^{\ast}(K_{X}+\Delta)$ for some divisor $\Delta_Y$ on $Y$ and $f^{\ast}|mD|=|M_{m}|+F_{m}$ where $|M_m|$ is free and $F_m$ is fixed part. Then, by construction, $\Supp\Delta_{Y}\cup \Supp F_{m}$ is a simple normal crossing divisor. Since $f^{\ast}L-\floor{\Delta_{Y}+\tfrac{1}{m}F_{m}}\sim_{\Q}K_{Y}+\frac{1}{m}M_{m}+\{\Delta_{Y}+\frac{1}{m}F_{m}\}$, Theorem \ref{thrm:Kodaira vanishing_semiample2} implies
$$ H^{q}(Y,\mc{O}_{Y}(f^{\ast}L-\floor{\Delta_{Y}+\tfrac{1}{m}F_{m}}))=0 $$
for $q>\dim X-\kappa(D)$. Therefore, using the Leray spectral sequence, we have
$$ H^{q}(X,\mc{O}_{X}(L)\otimes \mc{J}(X,\Delta,\norm{D}))=H^{q}(Y,\mc{O}_{Y}(f^{\ast}L-\floor{\Delta_{Y}+\tfrac{1}{m}F_{m}})=0 $$
for $q>\dim X-\kappa(D)$.
\end{proof}

\begin{corollary}\label{cor:main cor on vanishing}
Let $(X,\Delta,D)$ be a triple with a  big $\Q$-divisor $D$ on $X$. Suppose that $(X,\Delta,D)$ is pklt and $L$ a line bundle on $X$ such that $L\sim_{\Q}K_{X}+\Delta+D$. Then,
$$ H^{q}(X,\mc{O}_{X}(L))=0 $$
for any $q>0$.
\end{corollary}

Note that Corollary \ref{cor:main cor on vanishing} is a generalization of the Kawamata-Viehweg vanishing theorem and Theorem \ref{maintheorem3} is also a generalization of the main result of \cite{Wu} to the case of pairs.

\bigskip

\begin{corollary}
Let $(X,\Delta+M)$ be a generalized klt pair and $L$ a Cartier divisor on $X$ such that $L-(K_{X}+\Delta+M)$ is nef and big. Assume that $M$ is $\R$-Cartier. Then
$$ H^{q}(X,\mc{O}_{X}(L))=0 $$
for all $q>0$.
\end{corollary}

\begin{proof}
As we have seen in Section 3, the triple $(X,\Delta,M)$ is pklt. Let $A:=L-(K_{X}+\Delta+M)$. Since $A$ is nef, the triple $(X,\Delta,M+A)$ is also pklt. Furthermore, $M+A$ is big. Hence we can apply Corollary \ref{cor:main cor on vanishing} to the potential triple $(X,\Delta,M+A)$ and conclude that
$$ H^{q}(X,\mc{O}_{X}(L))=0 $$
for all $q>0$.
\end{proof}

\begin{remark}\label{remark: potential is general}
Corollary \ref{cor:main cor on vanishing} can be considered as a generalization of the main result of \cite{CLX} for the case of klt generalized pairs. The main result of \cite{CLX} in the generalized klt case implies that if $(X,\Delta+M_X)$ is a generalized klt pair and $L$ is a Cartier divisor such that $B:=L-(K_X+\Delta+M_X)$ is nef and big, then the vanishing $H^i(X,\mathcal{O}_X(L))=0$ holds for any $i>0$.  Under the same conditions, $M_X+B$ is big and the potential triple $(X,\Delta,M_X+B)$ is pklt. Thus the same vanishing can be obtained by applying Corollary \ref{cor:main cor on vanishing}.
\end{remark}



\begin{thebibliography}{ELMNP}
\bibitem[B]{B} C. Birkar, \textit{Generalised pairs in birational geometry}, EMS Surv. Math. Sci. 8 (2021), 5–24.

\bibitem[BCHM]{bchm} C. Birkar, P. Cascini, C. Hacon, J. McKernan,
\textit{Existence of minimal models for varieties of log general type},  J. Amer. Math. Soc. 23 (2010), no. 2, 405-468.

\bibitem[BZ]{BZ}  C. Birkar and D.-Q. Zhang, \textit{Effectivity of Iitaka fibrations and pluricanonical systems of polarized pairs}, Pub. Math. IHES. 123 (2016), 283-331.

\bibitem[C]{Cao}    J. Cao, \textit{Numerical dimension and a Kawamata-Viehweg-Nadel-type vanishing theorem on compact K\"ahler manifolds}, Compos. Math. 60 (2002), 295-313.

\bibitem[CD]{CD}  S. Cacciola, L. Di Biagio, \textit{Asymptotic base loci on singular varieties.} Math. Z. 275 (2013), no. 1-2, 151–166.

\bibitem[CLX]{CLX} B. Chen, J. Liu, L. Xie, \textit{Vanishing theorems for generalized pairs}, arXiv:2305.12337 [math.AG]


\bibitem[CJ]{CJ} S. Choi, S. Jang, \textit{ACC of plc thresholds}, to appear in Manuscripta Mathematica.


\bibitem[CJL]{CJL} S. Choi, S. Jang, D. Lee, \textit{On minimal model program and Zariski decomposition of potential triples}, 	arXiv:2502.00790 [math.AG], to appear in Taiwanese Journal of Mathematics.

\bibitem[CJKL]{CJKL} S. Choi, S. Jang, D. Kim, D. Lee, \textit{A valuative approach to the anticanonical minimal model program}, arXiv:2506.13637


\bibitem[CP]{CP} S. Choi, J. Park, \textit{Potentially non-klt locus and its applications.} Math. Ann. (2016) 366: 141-166  (see also \url{https://arxiv.org/abs/1412.8024} for update)

\bibitem[ELMNP]{ELMNP} L. Ein, R. Lazarsfeld, M. Musta\c t\u a, M. Nakamaye, M. Popa, \textit{Asymptotic invariants of base loci}, Annales de l'Institut Fourier, Volume 56 (2006) no. 6, pp. 1701-1734.

\bibitem[ELSV]{ELSV} L. Ein, R. Lazarsfeld, K.E. Smith, D.Varolin, \textit{Jumping coefficients of multiplier ideals}, Duke Math. J. 123 (2004), no. 3, 469–506.


\bibitem[GZ]{GZ} Q. Guan, X. Zhou, \textit{A proof of Demailly's strong openness conjecture}, Ann. of Math. (2) 182 (2015), no. 2, 605–616.

\bibitem[JM]{JM} M. Jonsson and M. Musta{\c{t}}{\u{a}}, \textit{Valuations and asymptotic invariants for sequences of ideals},
Ann. Inst. Fourier (Grenoble) \textbf{62} (2012), no. 6, 2145--2209.

\bibitem[Jow]{Jow} S-Y. Jow, \textit{Asymptotic order-of-vanishing functions on the pseudoeffective cone}, Pacific J. Math. \textbf{288} (2017), no. {2}, 377--380.

\bibitem[Kim]{Kim}
D. Kim,
\emph{On diminished multiplier ideal and the termination of flips},
arXiv preprint arXiv:2405.11902 (2024).

\bibitem[K1]{Kol1} J. Koll\'ar, \textit{Higher Direct Images of Dualizing Sheaves I}, Annals of Mathematics, Second Series, Vol. 123, No. 1 (Jan., 1986), pp. 11-42 .

\bibitem[KM]{km} J. Koll{\'{a}}r and S. Mori, \textit{Birational geometry of algebraic varieties}, Cambridge Tracts in Mathematics, vol. 134. Cambridge University Press, Cambridge, 1998.

\bibitem[KMM]{KMM} Y. Kawamata, K. Matsuda, K. Matsuki, \textit{Introduction to the minimal model problem}, Algebraic geometry, Sendai, 1985, 283–360.

\bibitem[KN]{KN}
L. Kuipers and H. Niederreiter,
\emph{Uniform distribution of sequences}, xiv+390pp, Wiley-Interscience [John Wiley \& Sons], New
              York-London-Sydney, (1974).

\bibitem[L1]{laza1} R. Lazarsfeld,
\textit{Positivity in algebraic geometry I},
48.  xviii+385pp. Springer-Verlag, Berlin, (2004).

\bibitem[L2]{laza2} R. Lazarsfeld,
\textit{Positivity in algebraic geometry \!II},
49. xviii+387pp. Springer-Verlag, Berlin, (2004).

\bibitem[Nak]{Nak} N. Nakayama, (2004). \textit{Zariski-decomposition and abundance} (Vol. 14). Tokyo: Mathematical Society of Japan.

\bibitem[Leh]{Leh}
B. Lehmann:
\emph{Algebraic bounds on analytic multiplier ideals},
Ann. Inst. Fourier (Grenoble)
\textbf{64}
(2014),
no. 3,
1077--1108.


\bibitem[LX]{LX}   C. Li, C. Xu, \textit{Stability of Valuations: Higher Rational Rank}. Peking Math J 1, 1–79 (2018).


\bibitem[PS]{PS}
Y. Prokhorov and V. Shokurov, \textit{Towards the second main theorem on complements} J. Algebraic Geom.  \textbf{18} (2009), 151--199.

\bibitem[W]{Wu} J. Wu, \textit{The Kawamata-Viehweg-Nadel-type vanishing theorem and the asymptotic multiplier ideal sheaf}, Math. Z. 302 (2022), no. 3, 1393–1407.

\bibitem[Xu]{Xu} C. Xu, \textit{A minimizing valuation is quasi-monomial}, Ann. of Math. (2) 191(3): 1003--1030 (May 2020).

\bibitem[Xu2]{Xu2} C. Xu, \textit{K-stability of Fano varieties}, volume 50 of New Mathematical Monographs. Cambridge University Press, Cambridge, 2025.


\end{thebibliography}
\end{document}